\documentclass[a4paper,12pt]{amsart}
\usepackage{amsmath, amssymb}
\usepackage[abbrev]{amsrefs}
\usepackage{color}
\usepackage{hyperref}
\hypersetup{colorlinks=true,linkcolor=blue,citecolor=blue}
\usepackage{comment}

\newtheorem{thm}{Theorem}[section]
\newtheorem{prop}[thm]{Proposition}
\newtheorem{lem}[thm]{Lemma}
\newtheorem{cor}[thm]{Corollary}
\newtheorem{clm}[thm]{Claim}

\theoremstyle{definition}
\newtheorem{dfn}[thm]{Definition}

\theoremstyle{remark}
\newtheorem{rem}[thm]{Remark}
\newtheorem{qst}[thm]{Question}

\numberwithin{equation}{section}

\newcommand{\R}{\ensuremath{\mathbb{R}}}
\newcommand{\X}{\mathcal{X}}

\newcommand{\cN}{\mathcal{N}}
\newcommand{\cP}{\mathcal{P}}
\newcommand{\cA}{\mathcal{A}}
\newcommand{\tX}{\tilde{X}}
\newcommand{\tY}{\tilde{Y}}
\newcommand{\cX}{\mathcal{X}}

\DeclareMathOperator{\supp}{supp}
\DeclareMathOperator{\diam}{diam}
\newcommand{\midd}{\mathrel{} \middle| \mathrel{}}

\newcommand{\kf}{d_{\mathrm{KF}}}
\newcommand{\prok}{d_{\mathrm{P}}}
\newcommand{\gp}{d_{\mathrm{GP}}}
\newcommand{\Lip}{{\mathcal{L}}ip}
\newcommand{\conc}{d_{\mathrm{conc}}}
\newcommand{\haus}{d_{\mathrm{H}}}
\newcommand{\tv}{d_{\mathrm{TV}}}

\newcommand{\Leb}{{\mathcal{L}}}

\DeclareMathOperator{\Sep}{Sep}
\DeclareMathOperator{\Fix}{Fix}

\newcommand{\ep}{\varepsilon}
\newcommand{\Dl}{\Delta}
\renewcommand{\phi}{\varphi}

\makeatletter
\@namedef{subjclassname@2020}{%
\textup{2020} Mathematics Subject Classification}
\makeatother

\title[Principal bundle structure of $\X$]{Principal bundle structure\\
of the space of metric measure spaces}

\thanks{This work was supported by JSPS KAKENHI Grant Numbers JP22K20338, JP22K13908, JP19K03459.}

\author[D.~Kazukawa]{Daisuke Kazukawa}
\address{Faculty of Mathematics, Kyushu University, Fukuoka 819-0395, JAPAN}
\email{kazukawa@math.kyushu-u.ac.jp}

\author[H.~Nakajima]{Hiroki Nakajima}
\address{Mathematical Sciences Course, Ehime University, Matsuyama 790-8577, JAPAN}
\email{nakajima.hiroki.nz@ehime-u.ac.jp}

\author[T.~Shioya]{Takashi Shioya}
\address{Mathematical Institute, Tohoku University, Sendai 980-8578, JAPAN}
\email{shioya@math.tohoku.ac.jp}

\date{April 13, 2023}

\begin{document}

\begin{abstract}
We study the topological structure of the space
$\cX$ of isomorphism classes of metric measure spaces
equipped with the box or concentration topologies.
We consider the scale-change action of the multiplicative
group $\R_+$ of positive real numbers on $\cX$,
which has a one-point metric measure space, say $*$,
as only one fixed-point.
We prove that the $\R_+$-action on $\cX_* := \cX \setminus \{*\}$
admits the structure of nontrivial and locally trivial
principal $\R_+$-bundle over the quotient space.
Our bundle $\R_+ \to \cX_* \to \cX_*/\R_+$ is
a curious example of a nontrivial principal fiber bundle
with contractible fiber.
A similar statement is obtained for the pyramidal compactification
of $\cX$, where we completely determine the structure
of the fixed-point set of the $\R_+$-action on the compactification.
\end{abstract}

\maketitle

\section{Introduction}

It is a challenging problem to study
the structure of the space $\cX$ of isomorphism classes
of metric measure spaces,
where we assume a metric measure space to be
a complete separable metric space with a Borel probability measure.
Denote by $\R_+$ the multiplicative group of positive real numbers.
We have the natural group action of $\R_+$ on $\cX$
defined as
\[
\R_+ \times \cX \ni (t,X) \longmapsto tX \in \cX,
\]
where $tX$ is the space $X$ with the $t$-scaled metric of $X$.
Note that the isomorphism class of a one-point metric measure
space, denoted as $*$, is the only fixed-point of this action.
Let $\cX_* := \cX \setminus \{*\}$ and
let $\Sigma$ denote the quotient space $\cX_*/\R_+$
equipped with the quotient topology.

As for the structure of the space $\cX$,
Sturm \cite{St} obtained
the remarkable result
that the subspace $\cX_{pq}$ of $\cX$
with finite $L^{pq}$-size and equipped with the $L^{p,q}$-distortion
metric is a nonnegatively curved Alexandrov space
isometric to a Euclidean cone for $p = 2$ and $q \in [\,1,+\infty\,)$.
He also determined geodesics in $\cX_{pq}$ for $p,q \in [\,1,+\infty\,)$,
and proved that any orbit of the $\R_+$-action is a geodesic ray,
which implies that $\cX_{pq}$ is homeomorphic to the cone over
$\Sigma_{pq}$ for any $p,q \in [\,1,+\infty\,)$,
where $\Sigma_{pq}$ is the subspace of $\Sigma$ corresponding to
$\cX_{pq}$.

Also, Ivanov and Tuzhilin \cite{IT} pointed out
that the Gromov-Hausdorff space
is homeomorphic to the cone over the quotient space
by the $\R_+$-action.

In this paper, we study the topological structure of $\cX$
with the box and observable metrics and also of the pyramidal compactification of $\cX$.
Those metrics and the pyramidal compactification are
fundamental concepts in the study of metric measure spaces,
originally introduced by Gromov \cite{Grmv} (see also \cite{MMG}).
The box metric is closely related with the $L^{0,q}$-distortion
metric (see \cite{St})
and coincides with the Gromov-Prokhorov metric (see \cite{GPW, Lohr}).
The observable metric is defined by the idea of
the concentration of measure phenomenon established by
L\'evy and Milman \cite{Levy, VMil} (see also \cite{Led})
and is useful to study convergence of
spaces with dimension divergent to infinity.
The topologies induced from the box and observable metrics
are called the box and concentration topologies, respectively.
In our previous paper \cite{KNS},
we have proved that the box metric is geodesic
and that $\cX$ is locally path connected and contractible
with respect to the box and concentration topologies.
However, the box and observable metrics are much different from
the $L^{p,q}$-distortion metric.
One of the essential differences is
that any orbit of the $\R_+$-action is never a geodesic ray
for the box and observable metrics.
Also, there are intricately branching geodesics
with respect to the box metric
and the Alexandrov curvature is not bounded neither from below
nor from above (see \cite{KNS}*{Remark 6.7}).
The question that arises here is:

\begin{itemize}
\item Is $\cX$ with the box and/or concentration topologies
homeomorphic to the cone over $\Sigma$?
\end{itemize}

Surprisingly the answer is negative.

\begin{thm} \label{thm:not-homeo}
For neither of the box nor concentration topologies,
$\cX_* = \cX \setminus \{*\}$
is not homeomorphic to $\R_+ \times \Sigma$,
and $\cX$ not homeomorphic to the cone over $\Sigma$.
\end{thm}

One of our main theorems is stated as follows.

\begin{thm} \label{thm:bundle}
For the box and concentration topologies,
the action of $\R_+$ on $\cX_*$ admits
the structure of a nontrivial and locally trivial
principal $\R_+$-bundle over $\Sigma$.
\end{thm}

In general, a locally trivial principal bundle with contractible fiber
over a paracompact Hausdorff base space is trivial
(see \cite{D}*{Corollary 2.8} and \cite{Huse}*{8.1 Theorem of Chapter 4} for example),
which is not necessarily true if the base space is not paracompact.
It is remarkable that our principal bundle $\R_+ \to \cX_* \to \Sigma$
presents such a counterexample.

The action of $\R_+$ on $\cX$ naturally extends to
the pyramidal compactification of $\cX$, say $\Pi$
(see Section \ref{ssec:pyramid} for the definition of $\Pi$).
Denote by $\Fix(\Pi)$ the set of fixed-points of
the action of $\R_+$ on $\Pi$,
and put $\Pi_* := \Pi \setminus \Fix(\Pi)$.
We also have the following.

\begin{thm} \label{thm:bundle-Pi}
The action of $\R_+$ on $\Pi_*$ admits
the structure of a nontrivial and locally trivial
principal $\R_+$-bundle over the quotient space $\Pi_*/\R_+$.
\end{thm}

We investigate the structure of $\Fix(\Pi)$
and have the following theorem.
Denote by $\cA$ the set of all monotone non-increasing sequences
of nonnegative real numbers with total sum not greater than one.

\begin{thm} \label{thm:Fix-A}
The fixed-point set $\Fix(\Pi)$ is homeomorphic to
$\cA$ with the $l^2$-weak topology.
\end{thm}

Let us mention the ideas of our proofs.

A key point of the proof of Theorem \ref{thm:not-homeo}
is to prove that $\Sigma$ is not a Urysohn space
(Lemma \ref{lem:Ury}).  If $\Sigma \times \R_+$ were to be
homeomorphic to $\cX_*$, then $\Sigma$ would be metrizable,
which is contrary to the non-Urysohn property of $\Sigma$.
It is quite delicate that $\Sigma$ is a Hausdorff space
(Proposition \ref{prop:Haus}).

For Theorems \ref{thm:bundle} and \ref{thm:bundle-Pi},
the local triviality of the bundles is a core of the proof.

For Theorem \ref{thm:bundle} with the box topology,
we construct an $\R_+$-invariant open covering
$\{\cX_\Delta\}_{\Delta \in (\,0,1\,)}$ of $\cX$
and continuous $1$-homogeneous functions
$r_\Delta \colon \cX_\Delta \to \R_+$ for $\Delta \in (\,0,1\,)$.
We define $\cX_\Delta$ to be the set of mm-spaces
such that any atom has measure less than $\Delta$,
and define $r_\Delta(X)$, $X \in \cX_\Delta$,
as the integral of the partial diameter $\diam(X;s)$
with respect to the parameter $s \in [\,0,(\Delta+1)/2\,]$.
The reason why we take the integral is for the sake of
the continuity of $r_\Delta$.
Using $r_\Delta$ we obtain a local trivialization
$\cX_\Delta \simeq \cX_\Delta/\R_+ \times \R_+$.

Theorem \ref{thm:bundle} for the concentration topology
is derived from Theorem \ref{thm:bundle-Pi}
just by restricting the base space $\Pi_*/\R_+$ to $\cX_*/\R_+$.

For the proof of Theorem \ref{thm:bundle-Pi},
we construct an $\R_+$-invariant open covering of $\Pi_*$
and continuous $1$-homogeneous functions on each open set
in the covering.
This time, for an $(N+1)$-tuple $\kappa = (\kappa_0,\dots,\kappa_N)$
of positive real numbers with $\sum_{i=0}^N \kappa_i < 1$,
we define $\Pi_\kappa$ to be the set of all $\cP \in \Pi$
such that  $\Sep(\cP;\kappa_0,\dots,\kappa_N) < +\infty$
and $\Sep(\cP;\kappa_0+\delta,\dots,\kappa_N+\delta) > 0$
for some $\delta > 0$, and
define $r_\kappa(\cP)$, $\cP \in \Pi_\kappa$, to be the integral of
$\Sep(\cP;\kappa_0+s,\dots,\kappa_N+s)$
with respect to $s \in [\,0,1\,]$
(see Definition \ref{dfn:sep} for the definition of $\Sep(\cdots)$).
These induce a local trivialization
$\Pi_\kappa \simeq \Pi_\kappa/\R_+ \times \R_+$.
However, it is not easy to prove that the union of all $\Pi_\kappa$
coincides with $\Pi_*$.
For the proof, we need to investigate the structure of pyramids
in $\Fix(\Pi)$ as follows.
For $A = \{a_i\}_{i=1}^\infty \in \cA$ we define
\[
\cP_A := \left\{X \in \X \midd \text{There exists a sequence $\{x_i\}_{i=1}^\infty \subset X$ such that } \sum_{i=1}^\infty a_i\delta_{x_i} \le \mu_X \right\}.
\]

\begin{thm} \label{thm:Fix}
For a given $\cP \in \Pi$,
the following {\rm (1)--(4)} are equivalent to each other.
\begin{enumerate}
\item $\cP \in \Fix(\Pi)$.
\item $t\cP = \cP$ for some $t \in \R_+$ with $t \neq 1$.
\item For any $\kappa_0,\dots,\kappa_N > 0$ with
$\sum_{i=0}^N \kappa_i < 1$,
the separation distance $\Sep(\cP;\kappa_0,\dots,\kappa_N)$
is either $0$ or $+\infty$.
\item There exists $A \in \cA$ such that
$\cP = \cP_A$.
\end{enumerate}
\end{thm}

In Theorem \ref{thm:Fix},
the implication `$(3) \Rightarrow (4)$' is highly nontrivial
and we need a delicate discussion to prove it.
Theorem \ref{thm:Fix} with a little discussion implies
that the union of $\Pi_\kappa$ coincides with $\Pi_*$.

Theorem \ref{thm:Fix-A} is derived from
Theorem \ref{thm:Fix} and the following.

\begin{thm} \label{thm:A}
The map $\cA \ni A \longmapsto \cP_A \in \Pi$
is an into homeomorphism.
\end{thm}

Theorem \ref{thm:A} is also proved by Esaki-Kazukawa-Mitsuishi
\cite{EKM} independently.  Our proof is simpler than \cite{EKM}.
It is proved in \cite{EKM} that the weak topology on $\cA$
coincides with the $l^\infty$-topology.

The organization of this paper is as follows.
After the preliminaries section,
we study in Section \ref{sec:scale-change}
the scale-change action of $\R_+$ on $\cX_*$.
We prove that $\Sigma$ is not Urysohn,
which leads to Theorem \ref{thm:not-homeo}.
We also prove Theorems \ref{thm:bundle} and \ref{thm:bundle-Pi}
with the help of Theorem \ref{thm:Fix}.
In Section \ref{sec:inv-pyramid},
we determine the structure of pyramids in $\Fix(\Pi)$
and prove Theorems \ref{thm:Fix} and \ref{thm:A}
to obtain Theorem \ref{thm:Fix-A}.
In the final Section \ref{sec:Questions},
we present several questions.

\section{Preliminaries}
In this section, we describe the definitions and some properties of metric measure space, the box distance, the observable distance, pyramid, and the weak topology. We use most of these notions along \cite{MMG}. As for more details, we refer to \cite{MMG} and \cite{Grmv}*{Chapter 3$\frac{1}{2}_+$}.

\subsection{Metric measure spaces}
Let $(X, d_X)$ be a complete separable metric space and $\mu_X$ a Borel probability measure on $X$. We call the triple $(X, d_X, \mu_X)$ a {\it metric measure space}, or an {\it mm-space} for short. We sometimes say that $X$ is an mm-space, in which case the metric and the measure of $X$ are respectively indicated by $d_X$ and $\mu_X$.

\begin{dfn}[mm-Isomorphism]
Two mm-spaces $X$ and $Y$ are said to be {\it mm-isomorphic} to each other if there exists an isometry $f \colon \supp{\mu_X} \to \supp{\mu_Y}$ such that $f_* \mu_X = \mu_Y$, where $f_* \mu_X$ is the push-forward measure of $\mu_X$ by $f$. Such an isometry $f$ is called an {\it mm-isomorphism}. Denote by $\mathcal{X}$ the set of mm-isomorphism classes of mm-spaces.
\end{dfn}

Note that an mm-space $X$ is mm-isomorphic to $(\supp{\mu_X}, d_X , \mu_X)$. We assume that an mm-space $X$ satisfies
\begin{equation*}
X = \supp{\mu_X}
\end{equation*}
unless otherwise stated.

\begin{dfn}[Lipschitz order]
Let $X$ and $Y$ be two mm-spaces. We say that $X$ ({\it Lipschitz}) {\it dominates} $Y$ and write $Y \prec X$ if there exists a $1$-Lipschitz map $f \colon X \to Y$ satisfying $f_* \mu_X = \mu_Y$. We call the relation $\prec$ on $\X$ the {\it Lipschitz order}.
\end{dfn}

The Lipschitz order $\prec$ is a partial order relation on $\X$.

\subsection{Box distance and observable distance}
For a subset $A$ of a metric space $(X, d_X)$ and for a real number $r > 0$, we set
\[
U_r(A) :=  \{x \in X \mid d_X(x, A) < r\},
\]
where $d_X(x, A) := \inf_{a \in A} d_X(x, a)$.

\begin{dfn}[Prokhorov distance]
The {\it Prokhorov distance} $\prok(\mu, \nu)$ between two Borel probability measures $\mu$ and $\nu$ on a metric space $X$ is defined to be the infimum of $\varepsilon > 0$ satisfying
\[
\mu(U_\varepsilon(A)) \geq \nu(A) - \varepsilon
\]
for any Borel subset $A \subset X$.
\end{dfn}

The Prokhorov metric $\prok$ is a metrization of the weak convergence of Borel probability measures on $X$ provided that $X$ is a separable metric space.

\begin{dfn}[Ky Fan metric]
Let $(X, \mu)$ be a measure space and $(Y, d_Y)$ a metric space. For two $\mu$-measurable maps $f,g \colon X \to Y$, we define $\kf^\mu (f, g)$ to be the infimum of $\varepsilon \geq 0$ satisfying
\begin{equation*}
\mu(\{x \in X \mid d_Y(f(x),g(x)) > \varepsilon \}) \leq \varepsilon.
\end{equation*}
The function $\kf^\mu$ is a metric on the set of $\mu$-measurable maps from $X$ to $Y$ by identifying two maps if they are equal to each other $\mu$-almost everywhere. We call $\kf^\mu$ the {\it Ky Fan metric}.
\end{dfn}

\begin{dfn}[Parameter]
Let $I := [0,1)$ and let $X$ be an mm-space. A map $\varphi \colon I \to X$ is called a {\it parameter} of $X$ if $\varphi$ is a Borel measurable map such that
\begin{equation*}
\varphi_\ast \mathcal{L}^1 = \mu_X,
\end{equation*}
where $\mathcal{L}^1$ is the one-dimensional Lebesgue measure on $I$.
\end{dfn}

Note that any mm-space has a parameter (see \cite{MMG}*{Lemma 4.2}).

\begin{dfn}[Box distance]
We define the {\it box distance} $\square(X, Y)$ between two mm-spaces $X$ and $Y$ to be the infimum of $\varepsilon \geq 0$ satisfying that there exist parameters $\varphi \colon I \to X$, $\psi \colon I \to Y$, and a Borel subset $I_0 \subset I$ with $\mathcal{L}^1(I_0) \geq 1 - \varepsilon$ such that
\begin{equation*}
|d_X(\varphi(s), \varphi(t)) - d_Y(\psi(s), \psi(t))| \leq \varepsilon
\end{equation*}
for any $s,t \in I_0$.
\end{dfn}

\begin{thm}[\cite{MMG}*{Theorem 4.10}]
The box distance function $\square$ is a complete separable metric on $\mathcal{X}$.
\end{thm}

Various distances equivalent to the box distance are defined and studied, for example, the Gromov-Prokhorov distance introduced by Greven-Pfaffelhuber-Winter \cite{GPW}.

\begin{thm}[\cite{Lohr}*{Theorem 3.1}, \cite{MMG}*{Remark 4.16}]
For any two mm-spaces $X$ and $Y$, we have
\[
\square(X, Y) = \gp((X, 2d_X, \mu_X), (Y, 2d_Y, \mu_Y)),
\]
where $\gp(X,Y)$ is the Gromov-Prokhorov metric defined to be the infimum of $\prok(\mu_X, \mu_Y)$ for all metrics on the disjoint union of $X$ and $Y$ that are extensions of $d_X$ and $d_Y$. In particular,
\[
\gp(X, Y) \leq \square(X, Y) \leq 2\gp(X, Y).
\]
\end{thm}

The topology induced from the box distance has historically various names, for example, the weak-Gromov topology.
However we call it the {\it box topology} in this paper.

The total variation distance is useful for estimating the box distance.

\begin{dfn}[Total variation distance]
The {\it total variation distance} $\tv(\mu, \nu)$ of two Borel probability measures $\mu$ and $\nu$ on a topological space $X$ is defined by
\[
\tv(\mu, \nu) := \sup_A{|\mu(A) - \nu(A)|},
\]
where $A$ runs over all Borel subsets of $X$.
\end{dfn}

If $\mu$ and $\nu$ are both absolutely continuous with respect to a Borel measure $\omega$ on $X$, then
\[
\tv(\mu, \nu)=\frac{1}{2}\int_X \left|\frac{d\mu}{d\omega} - \frac{d\nu}{d\omega}\right| \, d\omega,
\]
where $\frac{d\mu}{d\omega}$ is the Radon-Nikodym derivative of $\mu$ with respect to $\omega$.

\begin{prop}[\cite{MMG}*{Proposition 4.12}]
For any two Borel probability measures $\mu$ and $\nu$ on a complete separable metric space $X$, we have
\[
\square((X,\mu), (X,\nu)) \leq 2\prok(\mu,\nu) \leq 2\tv(\mu,\nu).
\]
\end{prop}

Given an mm-space $X$ and a parameter $\varphi \colon I \to X$ of $X$, we set
\[
\varphi^* \Lip_1(X) := \{ f \circ \varphi \mid f \colon X \to \R \text{ is $1$-Lipschitz} \},
\]
which consists of Borel measurable functions on $I$.

\begin{dfn}[Observable distance]
We define the {\it observable distance} $\conc(X, Y)$ between two mm-spaces $X$ and $Y$ by
\begin{equation*}
\conc(X, Y) := \inf_{\varphi, \psi} \haus(\varphi^* \Lip_1(X), \psi^* \Lip_1(Y)),
\end{equation*}
where $\varphi \colon I \to X$ and $\psi \colon I \to Y$ run over all parameters of $X$ and $Y$ respectively, and $\haus$ is the Hausdorff distance with respect to the metric $\kf^{\mathcal{L}^1}$.
\end{dfn}

\begin{thm}[\cite{MMG}*{Proposition 5.5 and Theorem 5.13}]
The observable distance function $\conc$ is a metric on $\mathcal{X}$. Moreover, for any two mm-spaces $X$ and $Y$,
\[
\conc(X, Y) \leq \square(X, Y).
\]
\end{thm}

We call the topology on $\X$ induced from $\conc$ the {\it concentration topology}.
We say that a sequence $\{X_n\}_{n=1}^\infty$ of mm-spaces  {\it concentrates} to an mm-space $X$ if $X_n$ $\conc$-converges to $X$ as $n \to \infty$.

\subsection{Pyramid} \label{ssec:pyramid}
\begin{dfn}[Pyramid] \label{dfn:pyramid}
A subset $\cP \subset \X$ is called a {\it pyramid} if it satisfies the following {\rm(1) -- (3)}.
\begin{enumerate}
\item If $X \in \cP$ and if $Y \prec X$, then $Y \in \cP$.
\item For any $Y, Y' \in \cP$, there exists $X \in \cP$ such that $Y \prec X$ and $Y' \prec X$.
\item $\cP$ is nonempty and $\square$-closed.
\end{enumerate}
We denote the set of all pyramids by $\Pi$. Note that Gromov's definition of a pyramid is only by (1) and (2). The condition (3) is added in \cite{MMG}.

For an mm-space $X$, we define
\begin{equation*}
\cP_X := \left\{Y \in \X \midd Y \prec X \right\},
\end{equation*}
which is a pyramid. We call $\cP_X$ the {\it pyramid associated with $X$}.
\end{dfn}

We observe that $Y \prec X$ if and only if $\cP_Y \subset \cP_X$. Note that $\X$ itself is a pyramid.

We define the weak convergence of pyramids as follows. This is exactly the Kuratowski-Painlev\'e convergence as closed subsets of $(\X, \square)$ (see \cite{KNS}*{Section 8}).

\begin{dfn}[Weak convergence]
Let $\cP$ and $\cP_n$, $n = 1, 2, \ldots$, be pyramids. We say that $\cP_n$ {\it converges weakly to} $\cP$ as $n \to \infty$ if the following (1) and (2) are both satisfied.
\begin{enumerate}
\item For any mm-space $X \in \cP$, we have
\begin{equation*}
\lim_{n \to \infty} \square(X, \cP_n) = 0.
\end{equation*}
\item For any mm-space $X \in \X \setminus \cP$, we have
\begin{equation*}
\liminf_{n \to \infty} \square(X, \cP_n) > 0.
\end{equation*}
\end{enumerate}
\end{dfn}

\begin{thm}[\cite{MMG}*{Section 6}]\label{thm:weak}
There exists a metric $\rho$ on $\Pi$ such that the following {\rm (1) -- (4)} hold.
\begin{enumerate}
\item $\rho$ is compatible with weak convergence.
\item $\Pi$ is $\rho$-compact.
\item The map
$\X \ni X \mapsto \cP_X \in \Pi$
is a $1$-Lipschitz topological embedding map with respect to $\conc$ and $\rho$.
\item The image of $\X$ is $\rho$-dense in $\Pi$.
\end{enumerate}
\end{thm}
In particular, $(\Pi,\rho)$ is a compactification of $(\X,\conc)$.
We call $(\Pi,\rho)$ the {\it pyramidal compactification} of $(\X,\conc)$.
We often identify $X$ with $\cP_X$, and we say that a sequence of mm-spaces {\it converges weakly} to a pyramid if the associated pyramid converges weakly.

\begin{dfn}[Approximation of a pyramid]
A sequence $\{Y_m\}_{m = 1}^\infty$ of mm-spaces is called an {\it approximation} of a pyramid $\cP$ provided that it satisfies
\[
Y_1 \prec Y_2 \prec \cdots \prec Y_m \prec \cdots \quad \text{ and } \quad \overline{\bigcup_{m = 1}^\infty \cP_{Y_m}}^{\, \square} = \cP.
\]
In particular, $\{Y_m\}_{m=1}^\infty$ converges weakly to $\cP$ as $m \to \infty$ and $Y_m \in \cP$ for all $m$.
\end{dfn}

It is known that any pyramid $\cP$ admits an approximation (see \cite{MMG}*{Lemma 7.14}).

\subsection{Separation distance}

The separation distance is one of the most fundamental invariants of an mm-space and a pyramid.

\begin{dfn}[Separation distance] \label{dfn:sep}
Let $X$ be an mm-space. For any real numbers $\kappa_0,\kappa_1,\ldots,\kappa_N > 0$ with $N \geq 1$, we define the {\it separation distance}
\[
\Sep(X;\kappa_0,\kappa_1,\ldots,\kappa_N)
\]
of $X$ as the supremum of $\min_{i\neq j} d_X(A_i,A_j)$
over all sequences of $N+1$ Borel subsets $A_0,A_1,\ldots,A_N \subset X$ satisfying $\mu_X(A_i) \geq \kappa_i$ for all $i=0,1,\ldots,N$. If $\kappa_i > 1$ for some $i$, then we define $\Sep(X;\kappa_0,\kappa_1,\ldots,\kappa_N) := 0$.
Moreover, we define the {\it separation distance of a pyramid} $\cP$ by
\[
\Sep(\cP;\kappa_0,\kappa_1,\ldots,\kappa_N):= \lim_{\delta \to 0+}\sup_{X \in \cP} \Sep(X;\kappa_0-\delta,\kappa_1-\delta,\ldots,\kappa_N-\delta) \ (\leq +\infty).
\]
\end{dfn}

The separation distance for mm-spaces is an invariant under mm-isomorphism. Note that
\[
\Sep(\cP_X;\kappa_0,\kappa_1,\ldots,\kappa_N) = \Sep(X;\kappa_0,\kappa_1,\ldots,\kappa_N)
\]
for any $\kappa_0,\kappa_1,\ldots,\kappa_N > 0$ and that $\Sep(\cP;\kappa_0,\kappa_1,\ldots,\kappa_N)$ is monotone non-increasing and left-continuous in $\kappa_i$ for each $i = 0,1,\ldots,N$, and that
\[
\Sep(\cP;\kappa_0,\kappa_1,\ldots,\kappa_N) \leq \Sep(\cP';\kappa_0,\kappa_1,\ldots,\kappa_N) \quad \text{if } \cP \subset \cP'.
\]

\begin{thm}[\cite{OS}*{Theorem 1.1}, Limit formula for separation distance]\label{thm:lim_form}
Let $\cP$ and $\cP_n$, $n = 1, 2, \ldots$, be pyramids. If $\cP_n$ converges weakly to $\cP$ as $n \to \infty$, then
\begin{align*}
\Sep(\cP;\kappa_0,\kappa_1,\ldots,\kappa_N) &= \lim_{\ep \to 0+}\liminf_{n \to \infty} \Sep(\cP_n;\kappa_0-\ep,\kappa_1-\ep,\ldots,\kappa_N-\ep) \\
& = \lim_{\ep \to 0+}\limsup_{n \to \infty} \Sep(\cP_n;\kappa_0-\ep,\kappa_1-\ep,\ldots,\kappa_N-\ep)
\end{align*}
for any $\kappa_0,\kappa_1,\ldots,\kappa_N > 0$.
\end{thm}

\section{Scale-change action} \label{sec:scale-change}

In this section, we prove Theorems \ref{thm:not-homeo}--\ref{thm:bundle-Pi}.

Let $\R_+ := (0,+\infty)$ be the multiplicative group of positive real numbers. We consider the scale-change action on $\X$;
\[
\R_+ \times \X \ni (t, X) \mapsto tX := (X, td_X, \mu_X) \in \X.
\]
The one-point space $*$ is the only fixed-point of this action and the set $\X_* := \X \setminus \{*\}$ is invariant.
The $\R_+$-action on $\X_*$ is free.
Let $\Sigma := \X_*/\R_+$ be the quotient space of $\X_*$ and $\pi \colon \X_* \to \Sigma$ the quotient map.
We denote the orbit $\pi(X)$ by $[X]$.

Simultaneously, we consider the scale-change action on $\Pi$;
\[
\R_+ \times \Pi \ni (t, \cP) \mapsto t\cP := \left\{tX\midd X\in\cP\right\} \in \Pi,
\]
which is a natural extension of the action on $\X$. Denote by $\Fix(\Pi)$ the set of fixed-points of this action, and put $\Pi_* := \Pi \setminus \Fix(\Pi)$. Then $\R_+$ acts on $\Pi_*$ freely.

For the proof of Theorem \ref{thm:not-homeo},
we need a lemma.

\begin{lem} \label{lem:nbd}
Let $Y_\varepsilon$ be the mm-space defined by
$Y_\varepsilon := (\{0,1\},|\cdot|,(1-\varepsilon)\delta_0+\varepsilon\delta_1)$ for $0 < \varepsilon < 1$.
Then, for any closed subset $V$ of $\Sigma$ with nonempty interior with respect to the box topology, there exists a $\delta(V) > 0$ such that
$[Y_\varepsilon]$ belongs to $V$ for any $\varepsilon$
with $0 < \varepsilon \le \delta(V)$.
\end{lem}

\begin{proof}
Let $V$ be a closed subset $V \subset \Sigma$ with nonempty interior.
Since any mm-space can be approximated by an mm-space with finite diameter, there is an mm-space $X$ with finite diameter such that
$[X]$ is an interior point of $V$.
Suppose that $[Y_\varepsilon]$ does not belong to $V$.
Then, since $Y_\varepsilon$ is an element in the open set $\pi^{-1}(\Sigma \setminus V)$, there is a large number $r_\varepsilon > 0$ such that the mm-space $Z_\varepsilon$ defined by
\[
Z_\varepsilon := X \sqcup \{z\}, \quad d_{Z_\varepsilon}|_{X\times X} := d_X, \quad d_{Z_\varepsilon}(z, X) := r_\varepsilon, \quad \mu_{Z_\varepsilon} := (1-\varepsilon)\mu_{X}+\varepsilon\delta_{z}
\]
satisfies $r_\varepsilon^{-1} Z_\varepsilon \in \pi^{-1}(\Sigma \setminus V)$.
Indeed, $\square(Y_\varepsilon, r_\varepsilon^{-1} Z_\varepsilon)$ is sufficiently small since $X \subset Z_\varepsilon$ is close to a one-point by scaling down $Z_\varepsilon$ with $r_\varepsilon^{-1}d_{Z_\varepsilon}(z, X) = 1$.
On the other hand, $Z_\varepsilon$ $\square$-converges to $X$ as $\varepsilon \to 0+$, which implies $[Z_\varepsilon] \in V$ for $\varepsilon > 0$ small enough.
This is a contradiction.
Thus, $[Y_\varepsilon]$ belongs to $V$ for every sufficiently small $\varepsilon > 0$.
This completes the proof.
\end{proof}

Lemma \ref{lem:nbd} implies the following.

\begin{lem}\label{lem:Ury}
For the quotient of the box topology on $\Sigma$,
any two distinct points in $\Sigma$ cannot be separated by
any closed neighborhoods.
In particular, $\Sigma$ is not a Urysohn space.
\end{lem}

\begin{proof}
We take two distinct points $[X], [X'] \in \Sigma$ and take any closed neighborhoods $V$, $V'$ of $[X]$, $[X']$, respectively.
Lemma \ref{lem:nbd} proves that
$[Y_\varepsilon]$ belongs to both $V$ and $V'$
for $0 < \varepsilon \le \min\{\delta(V),\delta(V')\}$.
This completes the proof.
\end{proof}

\begin{rem}
As is proved in Proposition \ref{prop:Haus} below, $\Sigma$ is Hausdorff.
In Lemmas \ref{lem:nbd} and \ref{lem:Ury}, to consider closed neighborhoods is essential.
\end{rem}

\begin{cor} \label{cor:Ury}
For the quotient of the concentration topology,
$\Sigma$ is not Urysohn.
Moreover, $\Pi_*/\R_+$ is not Urysohn.
\end{cor}

\begin{proof}
Since the quotient of the concentration topology is coarser than
that of the box topology on $\Sigma$,
Lemma \ref{lem:Ury} implies the first statement of the corollary.
Since $\Sigma$ is contained in $\Pi_*/\R_+$ as a subspace,
we obtain the second.
This completes the proof.
\end{proof}

\begin{proof}[Proof of Theorem \ref{thm:not-homeo}]
Suppose that $\X_*$ is homeomorphic to $\Sigma \times \R_+$.  Since $\X_*$ is a metric space, $\Sigma \simeq \Sigma \times \{1\}$ is metrizable, which contradicts Lemma \ref{lem:Ury}
for the box topology and Corollary \ref{cor:Ury}
for the concentration topology.
In the same way, $\X$ is not homeomorphic to the cone over $\Sigma$.
This completes the proof.
\end{proof}

In the same way as above, we see the following.

\begin{thm} \label{thm:not-homeo-Pi}
$\Pi_*$ is not homeomorphic to $(\Pi_*/\R_+) \times \R_+$.
\end{thm}

The rest of this section is devoted to prove
Theorems \ref{thm:bundle} and \ref{thm:bundle-Pi}.

We first assume that $\X$ is equipped with the box topology and prove Theorem \ref{thm:bundle} for the box topology.

\begin{prop}\label{prop:princ}
$\pi \colon \X_* \to \Sigma$ is a principal $\R_+$-bundle.
\end{prop}

\begin{proof}
We verify that if sequences $\{X_n\}_{n=1}^\infty \subset \X_*$ and $\{t_n\}_{n=1}^\infty \subset \R_+$ satisfy that $X_n$ and $t_nX_n$ $\square$-converge to $X \in \X_*$ and $tX$, respectively, as $n \to \infty$, then $t_n$ converges to $t$.
Since $\X \in \X_*$, there exist real numbers $\kappa_0, \ldots, \kappa_N > 0$ with $\sum_{i=0}^N \kappa_i < 1$ such that
\[
0 < \Sep(X; \kappa_0,\ldots,\kappa_N) < +\infty.
\]
Suppose that $t_n \not\to t$. There exist a real number $\delta>0$ and a subsequence $\{t_{n_k}\}_{k=1}^\infty$ such that either $t_{n_k} > t+\delta$ for any $k$ or $t_{n_k} < t-\delta$ for any $k$.
By applying Theorem \ref{thm:lim_form}, if $t_{n_k} > t+\delta$, then we have
\begin{align*}
&t\,\Sep(X; \kappa_0,\ldots,\kappa_N) = \Sep(tX; \kappa_0,\ldots,\kappa_N) \\
&=\lim_{\ep\to0+} \limsup_{k\to\infty} \Sep(t_{n_k}X_{n_k}; \kappa_0-\ep,\ldots,\kappa_N-\ep) \\
&\geq (t+\delta)\lim_{\ep\to0+} \limsup_{k\to\infty} \Sep(X_{n_k}; \kappa_0-\ep,\ldots,\kappa_N-\ep) \\
&= (t+\delta)\Sep(X; \kappa_0,\ldots,\kappa_N),
\end{align*}
which implies the contradiction $t \geq t+\delta$. Similarly, if $t_{n_k} < t-\delta$, then the contradiction $t \leq t-\delta$ holds. Thus we obtain $t_n \to t$. The proof is completed.
\end{proof}

Let us prove the local triviality of the principal fiber bundle
$\pi \colon \X_* \to \Sigma$.
Let $\Dl$ be a real number with $0<\Dl<1$ and put
\[
\X_\Dl := \left\{X \in \X \midd \mu_X(\{x\}) < \Dl \text{ for all } x \in X \right\}.
\]
We remark that $\sup_{x \in X} \mu_X(\{x\}) < \Dl$ if and only if $\diam(X; \Dl) > 0$,
where $\diam(X; \alpha)$ is the {\it partial diameter}, which is a fundamental invariant of an mm-space, given by
\[
\diam(X;\alpha) := \inf\left\{\diam{A} \midd A \text{ is a Borel subset with } \mu_X(A) \geq \alpha\right\}.
\]
We see that $\X_\Dl$ is open.
We have $\X_\Dl \subset \X_{\Dl'}$ for $\Dl \leq \Dl'$, and
\[
\X_* = \bigcup_{0<\Dl<1} \X_\Dl.
\]
Since $X \in \X_\Dl$ implies $tX \in \X_\Dl$ for any $t > 0$, the set $\X_\Dl$ is invariant with respect to the $\R_+$-action. Put $\Sigma_\Dl := \X_\Dl/\R_+$. Then, for the local triviality, it is sufficient to prove that $\X_\Dl$ is homeomorphic to $\Sigma_\Dl \times \R_+$ for every $\Dl \in (0,1)$.

We define a map $r_\Dl \colon \X_\Dl \to \R_+$ by
\[
r_\Dl(X) := \int_0^{\frac{\Dl+1}{2}} \diam(X;s) \, ds, \quad X \in \X_\Dl.
\]
Note that, for any $X \in \X_\Dl$,
\[
r_\Dl(X) \geq \frac{1-\Dl}{2}\diam(X; \Dl) > 0
\]
and that $r_\Dl(tX) = t \, r_\Dl(X)$ for any $t > 0$ since $\diam(tX; s) = t \diam(X; s)$ for any $s >0$.

\begin{lem}[\cite{MMG}*{Lemma 5.43}]\label{lem:partial}
If a sequence $\{X_n\}_{n=1}^\infty$ of mm-spaces $\square$-converges to an mm-space $X$, then we have
\[
\diam(X;s) \leq \liminf_{n\to\infty} \diam(X_n;s) \leq \limsup_{n\to\infty}\diam(X_n;s) \leq \lim_{\delta\to 0+} \diam(X;s+\delta)
\]
for any $s >0$.
\end{lem}

\begin{lem}
The map $r_\Dl$ is continuous on $\X_\Dl$.
\end{lem}

\begin{proof}
Take any sequence $\{X_n\}_{n=1}^\infty \subset \X_\Dl$ $\square$-converging to an mm-space $X \in \X_\Dl$. Let
\[
f_n(s) := \diam(X_n; s) \quad \text{ and } \quad f(s) := \diam(X; s)
\]
for $s \in [0, \frac{\Dl+1}{2}]$. Since $f$ is nondecreasing on $[0, \frac{\Dl+1}{2}]$, the discontinuous points of $f$ are at most countable. Thus Lemma \ref{lem:partial} implies $f_n$ converges almost everywhere to $f$ and
\[
\limsup_{n\to\infty} \sup_{s \in [0,\frac{\Dl+1}{2}]} f_n(s) \leq \diam(X; \frac{\Dl+3}{4}) < +\infty.
\]
Therefore, by the dominated convergence theorem,
\[
\lim_{n\to\infty}r_\Dl(X_n) = \lim_{n\to\infty}\int_0^{\frac{\Dl+1}{2}} f_n(s) \, ds = \int_0^{\frac{\Dl+1}{2}} f(s) \, ds = r_\Dl(X).
\]
The proof is completed.
\end{proof}

\begin{proof}[Proof of Theorem \ref{thm:bundle}
for the box topology]
The continuous $1$-homogeneous map $r_\Dl \colon \X_\Dl \to \R_+$ induces the homeomorphism
$\Phi \colon \X_\Dl \to \Sigma_\Dl \times \R_+$ defined by
\[
\Phi(X) := ([X], r_\Dl(X)) \quad \text{ for } X \in \X_\Dl.
\]
Indeed, the inverse map $\Phi^{-1}$ is given by
\[
\Phi^{-1}([X], t) = r_\Dl(X)^{-1}tX.
\]
In other words, the map $r_\Dl$ produces the continuous section
\[
\Sigma_\Dl \ni [X] \mapsto r_\Dl(X)^{-1} X \in \X_\Dl,
\]
so that $\X_\Dl \to \Sigma_\Dl$ is trivial.
This implies the local triviality of our principal fiber bundle
$\R_+ \to \cX_* \to \Sigma$.

Theorem \ref{thm:not-homeo} proves
that the fiber bundle $\R_+ \to \cX_* \to \Sigma$
is globally  nontrivial.
This completes the proof of Theorem \ref{thm:bundle}
for the box topology.
\end{proof}

The following corollary is a byproduct of Theorem \ref{thm:bundle}.

\begin{cor}
There is no continuous $1$-homogeneous map $r \colon \X_* \to \R_+$
with respect to the box topology on $\cX_*$.
\end{cor}

The following proposition is compared with Lemma \ref{lem:Ury}.

\begin{prop}\label{prop:Haus}
$\Sigma$ is a Hausdorff space.
\end{prop}

\begin{proof}
For any distinct two points $[X], [X'] \in \Sigma$, there exists $0<\Dl<1$ such that $[X], [X'] \in \Sigma_\Dl$. Since $\X_\Dl$ and $\Sigma_\Dl \times \R_+$ are homeomorphic, $\Sigma_\Dl$ is metrizable. Thus, the two points $[X], [X']$ can be separated by neighborhoods in $\Sigma$ since $\Sigma_\Dl$ is open in $\Sigma$.
The proof is completed.
\end{proof}

Before proving Theorem \ref{thm:bundle} for the concentration topology,
we study the scale-change action on $\Pi$ and prove Theorem \ref{thm:bundle-Pi} with the help of Theorem \ref{thm:Fix},
where  the proof of Theorem \ref{thm:Fix} is deferred to the next section.
The concentration case of Theorem \ref{thm:bundle}
 is obtained as a corollary of Theorem \ref{thm:bundle-Pi}.

\begin{prop}\label{prop:princ_py}
The quotient map $\pi \colon \Pi_* \to \Pi_*/\R_+$ is a principal $\R_+$-bundle.
\end{prop}

\begin{proof}
Assume that $\{\cP_n\}_{n=1}^\infty \subset \Pi_*$ and $\{t_n\}_{n=1}^\infty \subset \R_+$ satisfy that $\cP_n$ and $t_n\cP_n$ converge weakly to $\cP \in \Pi_*$ and $t\cP$, respectively, as $n \to \infty$.
Since $\cP \in \Pi_*$, there exist real numbers $\kappa_0, \ldots, \kappa_N > 0$ with $\sum_{i=0}^N \kappa_i < 1$ such that
\[
0 < \Sep(\cP; \kappa_0,\ldots,\kappa_N) < +\infty
\]
by Theorem \ref{thm:Fix}. By the same argument as in the proof of Proposition \ref{prop:princ}, we have $t_n \to t$ as $n\to\infty$.
The proof is completed.
\end{proof}

\begin{proof}[Proof of Theorem \ref{thm:bundle-Pi} under assuming Theorem \ref{thm:Fix}]
For any tuple $\kappa = (\kappa_0, \kappa_1, \ldots, \kappa_N)$ of positive real numbers with $\sum_{i=0}^N \kappa_i < 1$, we define
\[
\Pi_\kappa := \left\{\,\cP \in \Pi \midd \begin{array}{ll} \Sep(\cP; \kappa_0,\ldots,\kappa_N) <+\infty \text{ and } \\ \Sep(\cP; \kappa_0+\delta,\ldots,\kappa_N+\delta) > 0 \text{ for some } \delta > 0 \end{array}  \,\right\}.
\]
It is obvious that $\Pi_\kappa$ is invariant under $\R_+$-action.
We prove that $\Pi_\kappa$ is open.
Indeed, for any sequence $\{\cP_n\}_{n=1}^\infty$ of pyramids convergent weakly to a pyramid $\cP$, if $\Sep(\cP_n; \kappa_0,\ldots,\kappa_N) = +\infty$, then
\[
\Sep(\cP; \kappa_0,\ldots,\kappa_N) \geq \limsup_{n\to\infty} \Sep(\cP_n; \kappa_0,\ldots,\kappa_N) = +\infty
\]
by Theorem \ref{thm:lim_form}, and if $\Sep(\cP_n; \kappa_0+\delta,\ldots,\kappa_N+\delta) = 0$ for any $\delta>0$, then
\[
\Sep(\cP; \kappa_0+\delta,\ldots,\kappa_N+\delta) \le \liminf_{n\to\infty} \Sep(\cP_n; \kappa_0+\frac{\delta}{2},\ldots,\kappa_N+\frac{\delta}{2}) = 0
\]
for any $\delta > 0$ by Theorem \ref{thm:lim_form}. These imply that $\Pi_\kappa$ is open.

We define a map $r_\kappa \colon \Pi_\kappa \to \R_+$ by
\[
r_\kappa(\cP) := \int_0^1 \Sep(\cP; \kappa_0+s,\ldots,\kappa_N+s) \, ds, \quad \cP \in \Pi_\kappa.
\]
Then the map $r_\kappa$ is continuous and 1-homogeneous on $\Pi_\kappa$, so that $\Pi_\kappa \to \Pi_\kappa/\R_+$ is trivial.

The rest is to prove
\[
\Pi_* = \bigcup_\kappa \Pi_\kappa.
\]
By Theorem \ref{thm:Fix}, the inclusion $\bigcup_\kappa \Pi_\kappa \subset \Pi_*$ is obvious.
To prove the reverse inclusion
$\Pi_* \subset \bigcup_\kappa \Pi_\kappa$,
we take any $\cP \in \Pi_*$.
By Theorem \ref{thm:Fix},
there exist real numbers $\kappa_0, \kappa_1, \ldots, \kappa_N > 0$
with $\sum_{i=0}^N \kappa_i < 1$ such that
\[
0 < \Sep(\cP; \kappa_0, \ldots, \kappa_N) <+\infty.
\]
By the left-continuity of $\Sep(\cP;s_0,\ldots,s_N)$ in $s_i$, there is $\ep > 0$ such that
\[
\Sep(\cP; \kappa_0-\ep, \ldots, \kappa_N-\ep) <+\infty,
\]
which implies $\cP \in \Pi_{(\kappa_0-\ep, \ldots, \kappa_N-\ep)}$ and then $\Pi_* \subset \bigcup_\kappa \Pi_\kappa$.
Thus we obtain the local triviality of $\pi \colon \Pi_* \to \Pi_*/\R_+$.
The proof is completed.
\end{proof}

\begin{proof}[Proof of Theorem \ref{thm:bundle} for the concentration topology]
Assume that $\X$ is equipped with the concentration topology and consider the $\R_+$-action. Then $\pi \colon \X_* \to \X_*/\R_+$ is the restriction of bundle $\pi \colon \Pi_* \to \Pi_*/\R_+$. Thus it is also a principal $\R_+$-bundle and locally trivial (see \cite{Huse}).
This bundle is globally nontrivial because of Theorem \ref{thm:not-homeo}.
We finish the proof.
\end{proof}

\begin{cor}
\begin{enumerate}
\item
There is no continuous $1$-homogeneous map $r \colon \Pi_* \to \R_+$.
\item
There is no continuous $1$-homogeneous map $r \colon \cX_* \to \R_+$
with respect to the concentration topology on $\cX_*$.
\end{enumerate}
\end{cor}

\section{Scale-invariant pyramids} \label{sec:inv-pyramid}

In this section, we prove Theorems \ref{thm:Fix} and \ref{thm:A}.
We recall that $\cA$ is the set of all monotone non-increasing sequences of nonnegative real numbers with total sum not greater than one.
We equip $\cA$ with the weak topology as a closed convex subset of the space $l^2$.
In particular, $\cA$ is compact.
For every $A = \{a_i\}_{i=1}^\infty \in \cA$, we set
\[
\cP_A := \left\{X\in\X \midd \text{There exists a sequence }\{x_i\}_{n=1}^\infty \subset X \text{ such that } \sum_{i=1}^\infty a_i\delta_{x_i} \leq \mu_X \right\}.
\]
We remark that $\cP_A$ is a pyramid a priori, where the $\square$-closedness of $\cP_A$ follows from the argument in the proof of Claim \ref{clm:in} below.

\subsection{Characterization of $\Fix(\Pi)$}

In order to prove Theorem \ref{thm:Fix}, we need a lemma.

\begin{lem}\label{lem:dissip}
Let $a_1, \ldots, a_k$, $k<+\infty$, and $\ep$ be positive numbers with $\sum_{i=1}^k a_i < 1$ and $\min_{i=1,\ldots,k}a_i \geq \ep > 0$. If an mm-space $X$ admits distinct $k$ points $\{x_i\}_{i=1}^k$ with
\[
\mu_X(\{x\}) \begin{cases}
\geq a_i & \text{ if } x=x_i, \\
< \ep & \text{ if } x \neq x_i,
\end{cases}
\quad \text{and} \quad \sum_{i=1}^k (\mu_X(\{x_i\})-a_i) < \ep,
\]
then there exist positive numbers $\kappa_0,\ldots,\kappa_N > 0$ such that $\kappa_i \leq \ep$ for every $i$,
\[
0< 1-\sum_{i=1}^k a_i - \sum_{i=0}^N \kappa_i \leq \ep, \quad \text{and} \quad \Sep(X; a_1, \ldots, a_k, \kappa_0, \ldots, \kappa_N) > 0.
\]
\end{lem}

\begin{proof}
Take any mm-space $X$ having distinct $k$ points $\{x_i\}_{i=1}^k$ with
\[
\mu_X(\{x\}) \begin{cases}
\geq a_i & \text{ if } x=x_i, \\
< \ep & \text{ if } x \neq x_i.
\end{cases}
\]
Let $\{\xi_i\}_{i=1}^\infty$ be a countable dense subset of $X$ and let $d_i := d_X(\xi_i, \cdot \,)$. Put
\[
X_0 := X \setminus \bigcup_{i=1}^k U_{\ep'}(x_i)
\]
for some sufficiently small $\ep'>0$ with $\mu_X(X\setminus X_0) - \sum_{i=1}^k \mu_X(\{x_i\}) =: \eta_0 \ll \ep$.
We first find $\ep$-atomic points $\alpha_1, \ldots, \alpha_m$ of ${d_1}_* (\mu_X\lfloor_{X_0})$, i.e.,
\[
{d_1}_* (\mu_X\lfloor_{X_0})(\{t\}) \begin{cases}
\geq \ep & \text{ if } t=\alpha_i, \\
< \ep & \text{ if } t \neq \alpha_i,
\end{cases}
\]
if these exist.
We put
\[
X_i := d_1^{-1}(\{\alpha_i\})\cap X_0, \ i=1,\ldots,m, \quad \text{ and } \quad X^{(1)} := \bigcup_{i=1}^m X_i \subset X_0.
\]
There exist finitely many disjoint closed intervals $I_0, \ldots, I_{l_1}$
of $\R$ such that
\begin{align*}
&0 < {d_1}_* (\mu_X\lfloor_{X_0})(I_i) =: \kappa_i \leq \ep, \ i=0,\ldots,l_1, \\
&\text{and } \quad {d_1}_* (\mu_X\lfloor_{X_0})(\R\setminus(\{\alpha_1,\ldots,\alpha_{m}\}\sqcup \bigsqcup_{i=0}^{l_1} I_i)) =: \eta_1 \ll \ep.
\end{align*}
We put $A_i := d_1^{-1}(I_i) \cap X_0$. Note that $\mu_X(A_i) = \kappa_i$.

We next find $\ep$-atomic points $\alpha_{i1}, \ldots, \alpha_{im_i}$ of ${d_2}_* (\mu_X\lfloor_{X_i})$ for $i=1,\ldots,m$ if these exist, and we set
\[
X_{ij} := d_2^{-1}(\{\alpha_{ij}\})\cap X_i, \ j=1,\ldots,m_i, \quad \text{ and } \quad X^{(2)} := \bigcup_{i=1}^m \bigcup_{j=1}^{m_i} X_{ij} \subset X^{(1)}.
\]
There exist finitely many disjoint closed intervals $I_{l_1+1}, \ldots, I_{l_1+l_2}$ such that
\begin{align*}
&0 < {d_2}_* (\mu_X\lfloor_{X_{i(j)}})(I_j) =: \kappa_j \leq \ep, \ j=l_1+1,\ldots,l_1+l_2, \text{ for some } i(j) \in \{1,\ldots,m\},\\
&\text{and } \quad {d_2}_* (\mu_X\lfloor_{X^{(1)}})(\R\setminus(\{\alpha_{ij}\}_{i,j} \sqcup \bigsqcup_{j=l_1+1}^{l_1+l_2} I_j)) =:\eta_2 \ll \ep.
\end{align*}
We put $A_{j} := d_2^{-1}(I_j) \cap X_{i(j)}$ for every $j=l_1+1,\ldots,l_1+l_2$.

Repeating this construction, we obtain a monotone sequence $X^{(1)} \supset X^{(2)} \supset \cdots$ and a disjoint family $\{A_i\}$ on $X$.

We prove that $X^{(n)} = \emptyset$ for some $n$. It is sufficient to prove $\bigcap_{n=1}^\infty X^{(n)} = \emptyset$ since this implies $\mu_X(X^{(n)}) \to 0$ as $n \to \infty$, which contradicts the fact $\mu_X(X^{(n)}) \geq \ep$ if $X^{(n)} \neq \emptyset$.
Suppose that there exists a point $x_0 \in \bigcap_{n=1}^\infty X^{(n)} \neq \emptyset$. Then there exists a sequence $\{i_k\}_k$ such that $x_0 \in X_{i_1i_2\cdots i_k}$ for any $k$, where $i_k \in \{1, \ldots, m_{i_1i_2\cdots i_{k-1}}\}$.
Then it holds that
\[
\bigcap_{k=1}^\infty X_{i_1i_2\cdots i_k} = \{x_0\}.
\]
Indeed, suppose that there exists another point $y \in \bigcap_{k=1}^\infty X_{i_1i_2\cdots i_k}$ and let $r_0 := d_X(x_0, y)/4$. There exists a sufficiently large $k_0$ such that $d_X(\xi_{k_0}, x_0) \leq r_0$. Since $x_0, y \in X_{i_1i_2\cdots i_{k_0}}$, we have
\[
d_X(\xi_{k_0}, x_0) = d_X(\xi_{k_0}, y) = \alpha_{i_1i_2\cdots i_{k_0}},
\]
which implies the contradiction $4r_0 = d_X(x_0, y) \leq d_X(\xi_{k_0}, x_0) + d_X(\xi_{k_0}, y) \leq 2r_0$.

Thus we obtain
\[
\mu_X(\{x_0\}) = \mu_X(\bigcap_{k=1}^\infty X_{i_1i_2\cdots i_k}) = \lim_{k\to\infty} \mu_X(X_{i_1i_2\cdots i_k}) \geq \ep,
\]
but this contradicts the assumption of this lemma.

Therefore we have
\[
\Sep(X; a_1,\ldots,a_k,\kappa_0,\ldots,\kappa_{N}) \geq \min\{\min_{i\neq j}d_X(x_i,x_j), \min_{i, j}d_X(x_i,A_j), \min_{i\neq j}d_X(A_i,A_j)\} > 0,
\]
where $N = l_1+l_2+\cdots+l_n$, and
\begin{align*}
1 &= \mu_X(X_0) + \sum_{i=1}^k \mu_X(\{x_i\}) + \eta_0 \\
&= \mu_X(X^{(1)}) + \sum_{i=0}^{l_1}\kappa_i + \eta_1 + \sum_{i=1}^k \mu_X(\{x_i\}) + \eta_0 \\
&= \mu_X(X^{(2)}) + \sum_{i=0}^{l_1+l_2}\kappa_i + \sum_{i=0}^2 \eta_i + \sum_{i=1}^k \mu_X(\{x_i\}) \\
&= \mu_X(X^{(n)}) + \sum_{i=0}^{N}\kappa_i + \sum_{i=0}^n \eta_i + \sum_{i=1}^k \mu_X(\{x_i\}).
\end{align*}
Therefore, taking $\eta_i$ with $\sum_{i=0}^n \eta_i \leq \ep - \sum_{i=1}^k (\mu_X(\{x_i\})-a_i)$, we obtain the conclusion.
\end{proof}

\begin{proof}[Proof of Theorem \ref{thm:Fix}]
Since $(4) \Rightarrow (1) \Rightarrow (2) \Rightarrow (3)$ is trivial, we prove $(3) \Rightarrow (4)$. Assume that a pyramid $\cP$ satisfies the condition (3). Let $\{Y_n\}_{n=1}^\infty$ be an approximation of $\cP$ and let
\[
Y_\infty := \lim_{\longleftarrow} Y_n
\]
be the inverse limit of $\{Y_n\}_{n=1}^\infty$. There exists a probability measure $\mu_{Y_\infty}$ such that
\[
{\pi_n}_* \mu_{Y_\infty} = \mu_{Y_n}
\]
for any $n$, where $\pi_n \colon Y_\infty \to Y_n$ is the projection (see \cite{C}). Note that $Y_\infty$ admits an extended metric which is allowed to take values in $[0, +\infty]$ and $\pi_n$ is 1-Lipschitz.

Let $\{y_i\}_{i=1}^M \subset Y_\infty$, $M \leq +\infty$, be
the sequence of atomic points of $\mu_{Y_\infty}$ and let
\[
a_i := \mu_{Y_\infty}(\{y_i\})
\]
for every $i$. By relabeling, we can assume that
\[
a_1 \geq a_2 \geq \cdots \geq a_M \geq 0 =: a_{M+1} = a_{M+2} = \cdots
\]
and then $A := \{a_i\}_{i=1}^\infty \in \cA$. Then, for any $n$, we have
\[
\sum_{i=1}^M a_i \delta_{{\pi_n}(y_i)} = {\pi_n}_* (\sum_{i=1}^M a_i \delta_{y_i}) \leq {\pi_n}_* \mu_{Y_\infty} = \mu_{Y_n}.
\]

\begin{clm}\label{clm:in}
It holds that
\[
\cP \subset \cP_A.
\]
\end{clm}

\begin{proof}
Take any $X \in \cP$. Then there exist Borel maps $f_n \colon Y_n \to X$ such that $f_n$ is 1-Lipschitz up to $\ep_n$ with nonexceptional domain $\tY_n$, that is, $\mu_{Y_n}(\tY_n) \geq 1-\ep_n$ and
\[
d_X(f_n(y), f_n(y')) \leq d_{Y_n}(y,y') + \ep_n
\]
for any $y,y'\in\tY_n$, and $\prok({f_n}_* \mu_{Y_n}, \mu_X) \leq \ep_n$ for some $\ep_n \to 0$ (see \cite{comts}*{Lemma 4.6}).

We prove that, for each $i = 1,2,\ldots,M$, the sequence $\{f_n \circ \pi_n(y_i)\}_{n=1}^\infty \subset X$ has a convergent subsequence.
Suppose that $\{f_n \circ \pi_n(y_i)\}_{n=1}^\infty$ has no convergent subsequence. Then there exist a real number $\eta > 0$ and a subsequence $\{n_k\}_{k=1}^\infty$ of $\{n\}$ such that $\{B_\eta(f_{n_k} \circ \pi_{n_k}(y_i))\}_{k=1}^\infty$ is a disjoint family.
Since $\pi_{n_k}(y_i) \in \tY_{n_k}$ for sufficiently large $k$ and
\[
B_{\frac{\eta}{2}}(\pi_{n_k}(y_i)) \cap \tY_{n_k} \subset f_{n_k}^{-1}(B_{\frac{\eta}{2}+\ep_{n_k}}(f_{n_k} \circ \pi_{n_k}(y_i))),
\]
we have
\begin{align*}
0 < a_i & \leq \mu_{Y_{n_k}}(\{\pi_{n_k}(y_i)\}) \leq \mu_{Y_{n_k}}(B_{\frac{\eta}{2}}(\pi_{n_k}(y_i))) \\
&\leq {f_{n_k}}_* \mu_{Y_{n_k}}(B_{\frac{\eta}{2}+\ep_{n_k}}(f_{n_k} \circ \pi_{n_k}(y_i))) +\ep_{n_k} \leq \mu_{X}(B_\eta(f_{n_k} \circ \pi_{n_k}(y_i))) + 2\ep_{n_k}
\end{align*}
for sufficiently large $k$, which contradicts the disjointness of $\{B_\eta(f_{n_k} \circ \pi_{n_k}(y_i))\}_{n=1}^\infty$. Thus $\{f_n \circ \pi_n(y_i)\}_{n=1}^\infty$ has a convergent subsequence.

Let
\[
x_i := \lim_{k \to \infty} f_{n_k} \circ \pi_{n_k}(y_i),
\]
where $\{n_k\}$ is a subsequence of $\{n\}$ such that $\{f_{n_k} \circ \pi_{n_k}(y_i)\}_{k=1}^\infty$ converges for any $i=1,2,\ldots,M$ (by the diagonal argument). We prove
\[
\sum_{i=1}^M a_i \delta_{x_i} \leq \mu_X.
\]
For any nonnegative bounded continuous function $\phi\colon X \to [0, +\infty)$, by Fatou's lemma, we have
\[
\sum_{i=1}^M a_i \, \phi(x_i) \leq \liminf_{k\to\infty}\sum_{i=1}^M a_i \, \phi(f_{n_k} \circ \pi_{n_k}(y_i)) \leq \lim_{k \to \infty} \int_{Y_{n_k}} \phi \circ f_{n_k} \, d\mu_{Y_{n_k}} = \int_{X} \phi \, d\mu_{X},
\]
which implies $\sum_{i=1}^M a_i \delta_{x_i} \leq \mu_X$. Therefore the proof of this claim is completed.
\end{proof}

We next prove the converse inclusion under the condition (3).
We take any $\ep > 0$ and any positive integer $k$ such that
\[
k = \sup{\left\{i \midd a_i \geq \ep\right\}}.
\]
Then, for sufficiently large $n$, since $\pi_n \colon \{y_1, \ldots, y_k\} \to Y_n$ is injective, we have
\[
\mu_{Y_n}(\{y\}) \begin{cases}
\geq a_i & \text{ if } y=\pi_n(y_i), \\
< \ep & \text{ if } y \neq \pi_n(y_i),
\end{cases}
\quad \text{and} \quad \sum_{i=1}^k (\mu_{Y_n}(\{\pi_n(y_i)\})-a_i) < \ep.
\]
Indeed, if not, then the atomic part of $\mu_{Y_\infty}$ on the inverse limit $Y_\infty$ is not equal to $\sum_{i=1}^M a_i\delta_{y_i}$.
Thus, by Lemma \ref{lem:dissip}, there exist $\kappa_0,\ldots,\kappa_N > 0$ such that $\kappa_i \leq \ep$ for every $i$,
\[
0< 1-\sum_{i=1}^k a_i - \sum_{i=0}^N \kappa_i \leq \ep, \quad \text{and} \quad \Sep(Y_n; a_1, \ldots, a_k, \kappa_0, \ldots, \kappa_N) > 0
\]
for some large $n$. Thus we have
\[
\Sep(\cP; a_1, \ldots, a_k, \kappa_0, \ldots, \kappa_N) \geq \Sep(Y_n; a_1, \ldots, a_k, \kappa_0, \ldots, \kappa_N) > 0.
\]
Combining this and the condition (3) implies that
\[
\Sep(\cP; a_1, \ldots, a_k, \kappa_0, \ldots, \kappa_N) = \infty.
\]
By the limit formula of the separation distance, we have
\[
\lim_{\eta \to 0+}\lim_{n \to \infty} \Sep(Y_n; a_1-\eta, \ldots, a_k-\eta, \kappa_0-\eta, \ldots, \kappa_N-\eta) = \infty.
\]
Now we prove the following claim.
\begin{clm}\label{clm:ni}
It holds that
\[
\cP \supset \cP_A.
\]
\end{clm}
\begin{proof}
We take any mm-space $X$ admitting a sequence $\{x_i\}_{i=1}^\infty \subset X$ such that $\sum_{i=1}^\infty a_i\delta_{x_i} \leq \mu_X$. By the approximation, we can assume that $X$ is finite. Indeed, there exist finite nets $\cN_n$ of $X$ and Borel maps $\xi_n \colon X \to \cN_n$ such that $\lim_{n\to\infty}\prok({\xi_n}_* \mu_X, \mu_X)=0$.
Then we have
\[
\sum_{i=1}^\infty a_i\delta_{\xi_n(x_i)} \leq {\xi_n}_*\mu_X,
\]
so that $X$ can be approximated keeping our assumption. Let $\{z_1, \ldots, z_m\} := X$ and
\[
\nu_X := \mu_X - \sum_{i=1}^k a_i\delta_{x_i}.
\]

Since $\lim_{\eta \to 0+}\lim_{n \to \infty} \Sep(Y_n; a_1-\eta, \ldots, a_k-\eta, \kappa_0-\eta, \ldots, \kappa_N-\eta) = \infty$, there exist Borel subsets $Y_{n,1}, \ldots, Y_{n,N+k+1} \subset Y_n$ for any sufficiently large $n$ such that
\begin{align*}
&\mu_{Y_n}(Y_{n,i}) \geq a_i-\eta \text{ for } i=1,\ldots, k, \quad  \mu_{Y_n}(Y_{n,j+k+1}) \geq \kappa_j-\eta \text{ for } j=0,\ldots, N, \\
&\text{and} \quad \min_{i\neq j}d_{Y_n}(Y_{n,i}, Y_{n,j}) \geq \diam{X}
\end{align*}
for some $\eta < (N+k+1)^{-1}\ep$. We define a Borel map $g_n \colon Y_n \to X$ satisfying
\begin{align*}
&g_n(Y_{n,i}) = x_i \text{ for } i=1,\ldots,k \quad \text{ and } \quad g_n(Y_{n,j+k+1}) = z_l \text{ for } j_{l-1} \leq j < j_l,
\end{align*}
where
\[
j_0 := 0 \quad \text{and} \quad j_l := \max\left\{j \geq j_{l-1} \midd \nu_X(\{z_l\}) \geq \sum_{i = j_{l-1}}^{j-1} \kappa_i  \right\} \text{ for } l=1,\ldots,m.
\]
Letting $\tY_n := \bigcup_{i=1}^{j_{m}+k} Y_{n,i}$, by the definition of $g_n$, we have
\[
0 \leq \mu_X(\{z_l\}) - {g_n}_*(\mu_{Y_n}\lfloor_{\tY_n})(\{z_l\}) \leq 2\ep
\]
for any $l=1,\ldots,m$. In particular, $1-\mu_{Y_n}(\tY_n) \leq 2m\ep$. Moreover, for any $B \subset X$, we have
\[
{g_n}_* \mu_{Y_n}(B) \leq {g_n}_* (\mu_{Y_n}\lfloor_{\tY_n})(B) + 2m\ep \leq \mu_X(B) + 2m\ep.
\]
Thus $g_n$ is 1-Lipschitz up to $2m\ep$ with nonexceptional domain $\tY_n$ and
$\prok({g_n}_* \mu_{Y_n}, \mu_X) \leq 2m\ep$.
Since $Y_n \in \cP$, taking $\ep \to 0$, we obtain $X \in \cP$ (see \cite{comts}*{Corollary 4.7}).
\end{proof}
We finish the proof of this theorem.
\end{proof}

\subsection{Topological structure of $\Fix(\Pi)$}

The goal here is to prove Theorem \ref{thm:A}.

\begin{lem}\label{lem:inj}
The map $\cA \ni A \mapsto \cP_A \in \Fix(\Pi)$ is injective.
\end{lem}

\begin{proof}
Take any $A=\{a_i\}_{i=1}^\infty, A'=\{a'_i\}_{i=1}^\infty \in \cA$ with $A \neq A'$.
There exists a number $k$ such that $a_i=a'_i$ for any $i<k$ and $a_k \neq a'_k$. We can assume that $a_k < a'_k$. An mm-space $X$ is defined as the unit interval $([0,1], |\cdot|)$ with probability measure
\[
\mu_X := \sum_{i=1}^\infty a_i \delta_{2^{-i}} + \left(1-\sum_{i=1}^\infty a_i\right) \Leb^1,
\]
where $\Leb^1$ is the Lebesgue measure on $[0,1]$.
Then we have $X \in \cP_A$ and $X \not\in \cP_{A'}$. Indeed, if $X \in \cP_{A'}$, then there exists a sequence $\{x_i\}_{i=1}^\infty \subset X$ such that $\sum_{i=1}^\infty a'_i\delta_{x_i} \leq \mu_X$.
Since $A$ is monotone non-increasing, we have
\[
\sum_{i=1}^k a'_i \leq \mu_X(\{x_1,\ldots,x_k\}) \leq \sum_{i=1}^k a_i < \sum_{i=1}^k a'_i,
\]
which is a contradiction. Therefore we obtain $\cP_A \neq \cP_{A'}$. This completes the proof.
\end{proof}

\begin{lem}\label{lem:conti}
The map $\cA \ni A \mapsto \cP_A \in \Fix(\Pi)$ is continuous.
\end{lem}

\begin{proof}
Assume that $A_n = \{a_{ni}\}_{i=1}^\infty \in \cA$ converges weakly to $A = \{a_i\}_{i=1}^\infty \in \cA$. Let us prove that $\cP_{A_n}$ converges weakly to $\cP_A$.

We first prove that $\lim_{n\to\infty} \square(X, \cP_{A_n}) = 0$ for any $X \in \cP_A$. By the standard approximation, $X$ can be assumed to be a finite mm-space. Take any $\ep > 0$ and find a number $k$ such that $a_{k+1} < \ep$. Then, for sufficiently large $n$, we have
\[
|a_{ni}-a_i| < \frac{\ep}{2^i} \text{ for } i=1,\ldots,k \quad \text{and} \quad a_{n,k+1} < \ep.
\]
Since $A_n$ is a monotone non-increasing sequence, $a_{n,k+1} < \ep$ implies $\sup_{i>k} a_{ni} < \ep$. Take such large $n$ and fix it. Let $\{x_i\}_{i=1}^\infty$ be a sequence in $X$ such that $\sum_{i=1}^\infty a_i\delta_{x_i} \leq \mu_X$ and let $\{y_1,\ldots,y_N\} := X$.
We define
\[
\tX := \{\xi_1,\ldots,\xi_k, \eta_1,\ldots,\eta_N\}
\]
and define two maps $\phi\colon\tX\to X$ and $\psi\colon X \to\tX$ by
\[
\phi(x) := \begin{cases}
x_i & \text{if } x=\xi_i, \\
y_i & \text{if } x=\eta_i,
\end{cases}
\quad \text{and} \quad
\psi(y_i) := \eta_i.
\]
We now define two probability measures $\mu_{\tX}$ and $\mu_{\tX_n}$ on $\tX$ as
\[
\mu_{\tX} := \sum_{i=1}^k a_i \delta_{\xi_i} + \psi_*(\mu_X-\sum_{i=1}^k a_i\delta_{x_i}), \quad \mu_{\tX_n}\lfloor_{\{\xi_1, \ldots, \xi_k\}} := \sum_{i=1}^k a_{ni} \delta_{\xi_i}
\]
and $\mu_{\tX_n}\lfloor_{\{\eta_1, \ldots, \eta_N\}}$ is determined as follows. Find finitely many real numbers $b_{n1},\ldots,b_{nM} \in [0,\ep)$ with
\[
\sum_{i=1}^M b_{ni} = 1-\sum_{i=1}^\infty a_{ni}
\]
and set
\[
c_{nj} := \begin{cases}
b_{nj} & \text{if } 1\leq j\leq M, \\
a_{n, j-M+k} & \text{if } j>M.
\end{cases}
\]
Note that $\sup_{j} c_{nj} < \ep$. We define
\[
\mu_{\tX_n}(\{\eta_i\}) := \sum_{j=j_{i-1}+1}^{j_i} c_{nj}
\]
for $i=1,\ldots,N$, where $j_0 := 0$, $j_N := +\infty$, and
\[
j_i := \inf\left\{j > j_{i-1} \midd \sum_{l = j_{i-1}+1}^j c_{nl} \geq \mu_{\tX}(\{\eta_i\}) \right\}
\]
for $i=1,\ldots,N-1$.
Under $\inf{\emptyset} = +\infty$, if there exists $i_0<N$ such that $j_{i_0} = \cdots =j_N = +\infty$, then we understand
\[
\mu_{\tX_n}(\{\eta_{i_0+1}, \ldots, \eta_{N}\}) = 0.
\]
Letting $i_0 := \min\left\{1\leq i \leq N \midd j_i=+\infty \right\}$,
we have
\[
\mu_{\tX}(\{\eta_i\}) \leq \mu_{\tX_n}(\{\eta_i\}) \leq \mu_{\tX}(\{\eta_i\})+\ep
\]
for any $i < i_0$ by the definition. On the other hand, since
\[
1-\sum_{i=1}^k a_{ni} = \sum_{j=1}^\infty c_{nj} \leq \sum_{i=1}^{i_0} \mu_{\tX}(\{\eta_i\}) + (i_0-1)\ep,
\]
we have
\[
\sum_{i=i_0+1}^N \mu_{\tX}(\{\eta_{i}\}) = 1-\sum_{i=1}^k a_{i} - \sum_{i=1}^{i_0} \mu_{\tX}(\{\eta_i\}) \leq \sum_{i=1}^k|a_{ni}-a_i| + (i_0-1)\ep \leq i_0\ep.
\]
These imply that
\begin{align*}
&|\mu_{\tX_n}(\{\eta_{i_0}\}) - \mu_{\tX}(\{\eta_{i_0}\})| \\
&\leq \sum_{i=1}^{i_0-1}|\mu_{\tX_n}(\{\eta_i\}) - \mu_{\tX}(\{\eta_i\})| + \sum_{i=i_0+1}^N \mu_{\tX}(\{\eta_{i}\}) + \sum_{i=1}^k |a_{ni}-a_i| \\
&\leq 2i_0\ep \leq 2N\ep.
\end{align*}
Hence we have
\[
\tv(\mu_{\tX_n}, \mu_{\tX}) = \frac{1}{2} \sum_{i=1}^k |a_{ni}-a_i| + \frac{1}{2} \sum_{i=1}^N |\mu_{\tX_n}(\{\eta_i\})-\mu_{\tX}(\{\eta_i\})| \leq 2N\ep.
\]
Therefore, letting $X_n := (X, d_X, \phi_* \mu_{\tX_n})$, we obtain $X_n \in \cP_{A_n}$ and
\[
\square(X, X_n) \leq 2\tv(\mu_X, \phi_*\mu_{\tX_n}) \leq 2\tv(\mu_{\tX}, \mu_{\tX_n}) \leq 4N\ep,
\]
which imply $\lim_{n\to\infty} \square(X, \cP_{A_n}) = 0$.

We next prove that $\liminf_{n\to\infty} \square(X, \cP_{A_n}) > 0$ for any $X \in \X\setminus\cP_A$. It is sufficient to prove that if $X_n \in \cP_{A_n}$ $\square$-converges to $X$, then we have $X \in \cP_A$, due to consider the contraposition and to extract a subsequence. Assume that $X_n \in \cP_{A_n}$ $\square$-converges to $X$.
Let $\{x_{ni}\}_{i=1}^\infty$ be a sequence in $X_n$ with
\[
\sum_{i=1}^\infty a_{ni}\delta_{x_{ni}} \leq \mu_{X_n}.
\]
There exist Borel maps $f_n \colon X_n \to X$ and a sequence $\ep_n \to 0$ such that $f_n$ is 1-Lipschitz up to $\ep_n$ and $\prok({f_n}_* \mu_{X_n}, \mu_X) \leq \ep_n$ (actually, $f_n$ can be assumed to be an $\ep_n$-mm-isomorphism but this is unnecessary here). The sequence $\{f_n(x_{ni})\}_{n=1}^\infty$ has a convergent subsequence by the same argument as in the proof of Claim \ref{clm:in}. Let
\[
x_i := \lim_{k \to \infty} f_{n_k}(x_{n_ki}),
\]
where $\{n_k\}$ is a subsequence of $\{n\}$ such that $\{f_{n_k}(x_{n_ki})\}_{k=1}^\infty$ converges for any $i$. Then, since $A_n$  converges weakly to $A$ and ${f_n}_* \mu_{X_n}$ converges weakly to $\mu_X$, we have
\[
\sum_{i=1}^\infty a_i \, \phi(x_i) \leq \liminf_{k\to\infty}\sum_{i=1}^\infty a_{n_ki} \, \phi(f_{n_k}(x_{n_ki})) \leq \lim_{k \to \infty} \int_{X_{n_k}} \phi \circ f_{n_k} \, d\mu_{X_{n_k}} = \int_{X} \phi \, d\mu_{X}
\]
for any nonnegative bounded continuous function $\phi\colon X \to [0, +\infty)$, where the first inequality follows from Fatou's lemma. This implies $\sum_{i=1}^\infty a_i \delta_{x_i} \leq \mu_X$, so that $X \in \cP_A$. The proof of this lemma is now completed.
\end{proof}

\begin{proof}[Proof of Theorem \ref{thm:A}]
By Lemmas \ref{lem:inj} and \ref{lem:conti}, the map $\cA \ni A \mapsto \cP_A \in \Fix(\Pi)$ is a continuous bijection from the compact space $\cA$ to the Hausdorff space $\Fix(\Pi)$. Thus this is a homeomorphism. The proof is completed.
\end{proof}

\section{Further questions} \label{sec:Questions}

It is asked in \cite{BCZ}*{Question 9.1}
if the Gromov-Hausdorff space is homeomorphic to $l^2$.
In our previous paper \cite{KNS},
we have proved that $\cX$ is not homeomorphic to $l^2$
with respect to the concentration topology.
The following question remains.

\begin{qst}
Is $\cX$ homeomorphic to $l^2$ with respect to the box topology?
\end{qst}

It follows from Theorem \ref{thm:bundle} (resp.~\ref{thm:bundle-Pi})
that $\X_*$ and $\Sigma$ (resp.~$\Pi_*$ and $\Pi_*/\R$)
are weakly homotopy equivalent to each other.
We also ask the following.

\begin{qst} \label{qst:contra}
Is the space $\cX_*$ contractible and/or locally contractible
with respect to the box and/or concentration topologies?
\end{qst}

\begin{qst} \label{qst:contra-quot}
Is $\Sigma$ contractible with respect to the quotient of the box and/or concentration topologies?
\end{qst}

\begin{qst}
What about the same questions as Questions
\ref{qst:contra} and \ref{qst:contra-quot}
for $\Pi$, $\Pi_*$, and $\Pi_*/\R_+$
instead of $\X$, $\X_*$, and $\Sigma$?
\end{qst}

The following is already stated in our previous paper \cite{KNS}.

\begin{qst} \label{qst:geod}
Is the observable metric on $\cX$ geodesic?
Is the pyramidal compactification $\Pi$ of $\cX$ geodesic?
\end{qst}

\subsection*{Acknowledgements}
The authors would like to thank Professors Theo Sturm, Max von Renesse, Ayato Mitsuishi, Atsushi Katsuda, Kohei Suzuki,
and Shouhei Honda for their valuable comments.

\end{document}